\documentclass[12pt]{article}
\usepackage[utf8]{inputenc}
\usepackage{amsmath}
\usepackage{graphicx}
\usepackage{amssymb}
\numberwithin{equation}{section}
\usepackage{mathtools}
\usepackage{physics}
\usepackage{enumerate}
\usepackage{dashrule}

\usepackage[page, toc]{appendix}

\usepackage				[usenames,dvipsnames]	{color} 
\usepackage										{textcomp}
\usepackage										{leftidx}
\usepackage										{ulem}
\usepackage										{bbding}
\usepackage										{diagbox}
\usepackage										{multirow}
\usepackage										{longtable}
\usepackage										{listings}
\usepackage										{xparse}
\usepackage										{ifmtarg}
\usepackage										{mleftright}
\usepackage										{xcolor}
\usepackage										{pdflscape}
\usepackage										{bbm}
\usepackage										{amsbsy}
\usepackage										{pdfpages}
\usepackage										{subcaption}
\usepackage				[page, toc]				{appendix}
\usepackage										{siunitx}
\usepackage										{float}
\usepackage										{array}
\usepackage	[notextcomp, notext, nomath, not1]	{stix}
\usepackage										{dsfont}
\usepackage										{bigints}
\usepackage										{amsthm}
		\newtheorem{theorem}		{Theorem}		[section]
		\newtheorem*{theorem*}		{Theorem}
		\newtheorem{maintheorem}	{Theorem}	[]
		\newtheorem*{maintheorem*}	{Theorem}
		
		\newtheorem{mainconjecture}	{Conjecture}	[]
		\newtheorem*{conjecture*}	{Conjecture}
		
		\newtheorem{proposition}	{Proposition} 	[section]
		\newtheorem{lemma}			{Lemma} 		[proposition]
	\theoremstyle{definition}
		\newtheorem{definition}		{Definition}	[section]
	\theoremstyle{remark}
		\newtheorem{remark}			{Remark}		[theorem]
		
		\newtheorem*{eg*}			{Example}
	\theoremstyle{plain}
		
\numberwithin{table}{section}

\usepackage										{jheppub}
	\definecolor{darkblue}{rgb}{0.0,0.0,0.4}
	\definecolor{darkgreen}{rgb}{0.0,0.4,0.0}
	\hypersetup{
		colorlinks,
		linkcolor=black,
		citecolor=darkgreen,
		urlcolor=darkblue
	}

\usepackage{tikz-cd}
\newsavebox{\verticalseparator}
\sbox{\verticalseparator}{\tikz {\path [fill, draw] (0,0) [out=0, in=180] to +(.25\textwidth,1pt) [out=0, in=180] to +(.5\textwidth,0pt) [out=180, in=0] to +(.25\textwidth,-1pt) [out=180, in=0] to cycle;}}

\def\red{\color{red}}
\def\black{\color{black}}

\def\gray{\color{gray}}

\def\white{\color{white}}

\def\CC{{\mathbb C}}

\def\LL{{\mathbb L}}

\def\NN{{\mathbb N}}

\def\QQ{{\mathbb Q}}
\def\RR{{\mathbb R}}

\def\ZZ{{\mathbb Z}}

\def\Hh{{\mathcal H}}

\def\Ll{{\mathcal L}}
\def\Mm{{\mathcal M}}

\def\Oo{{\mathcal O}}
\def\Pp{{\mathcal P}}

\def\Yy{{\mathcal Y}}

\makeatletter

\NewDocumentCommand{\Mgn}{oo}{%
	\IfValueTF{#1}%
		{\IfValueTF{#2}%
			{\@ifmtarg{#1}{\Mm^{#2}}{\Mm^{#2}_{#1}}}%
			{\@ifmtarg{#1}{}{\Mm_{#1}}}%
		}%
		{\IfValueTF{#2}%
			{\@iftarg{#2}{}{\Mm^{#2}}}%
			{\Mm_{g, n}}
		}
}
\NewDocumentCommand{\Mgnbar}{oo}{%
	\IfValueTF{#1}%
	{\IfValueTF{#2}%
		{\@ifmtarg{#1}{\overline{\Mm}^{#2}}{\overline{\Mm}^{#2}_{#1}}}%
		{\@ifmtarg{#1}{}{\overline{\Mm}_{#1}}}%
	}%
	{\IfValueTF{#2}%
		{\@ifmtarg{#2}{}{\overline{\Mm}^{#2}}}%
		{\overline{\Mm}_{g, n}}
	}
}

\DeclareMathOperator*{\Residuum}{Res}
\def\Res[#1]{{\underset{#1}{\Residuum}}}
\NewDocumentCommand{\Wgn}{o o}%
	{\IfValueTF{#1}%
		{\IfValueTF{#2}%
			{W_{#2}^{(#1)}}%
			{W^{(#1)}}%
		}%
	{W_n^{(g)}}%
	}

\NewDocumentCommand{\bbrack}{o}{\IfValueTF{#1}{[\![#1]\!]}{[\![]\!]}}
\NewDocumentCommand{\angles}{o o}%
	{\IfValueTF{#1}%
		{\IfValueTF{#2}%
			{\@ifmtarg{#1}{\langle \cdot \rangle_{#2}}{\langle #1 \rangle_{#2}}}%
			{\@ifmtarg{#1}{\langle \cdot \rangle}{\langle #1 \rangle}}%
		}%
		{\IfValueTF{#2}%
			{\@ifmtarg{#1}{}{\langle \cdot \rangle_{#2}}}%
			{\langle \cdot \rangle}%
		}%
	}
\def\angles[#1, #2]{{\langle#1\rangle_{{#2}}}}
\NewDocumentCommand{\sumkappa}{o}%
	{\IfValueTF{#1}%
		{e^{\sum_k \tilde{t}_{a_{#1}, k} \kappa_k}}%
		{e^{\sum_k \tilde{t}_{a, k} \kappa_k}}
	}

\NewDocumentCommand{\ceil}{m}{\left\lceil#1\right\rceil}

\def\mathi{{\mathrm{i}}}

\def\conv[#1]{{\textnormal{Conv}\mleft(#1\mright)}}

\makeatother

\usepackage{tikz-cd}
\usepackage[runin]{abstract}
\abslabeldelim{.}

\date{date}

\begin{document}

\begin{titlepage} 
	\begin{center}
	\vspace{1cm}   
	\vspace*{ 2.0cm}
	{\Large {\bf Tautological Intersection Numbers and Order-Consecutive Partition Sequences}}\\[12pt]
	\vspace{-0.1cm}
	\bigskip
	\bigskip 
	{ {{Finn Bjarne Jost}$^{\,\text{a, b}}$}
	\bigskip }\\[3pt]
	\vspace{0.cm}
	{
	  ${}^{\text{a}}$ 
	  {\it
		  Institute for Theoretical Physics,~Westfälische Wilhelms-Universität Münster, Wilhelm-Klemm-Straße 9, 48149 Münster, Germany\\
	  }
	  ${}^{\text{b}}$ 
	  {\it
		  Mathematical Institute,~Westfälische Wilhelms-Universität Münster,\\ Einsteinstraße 62, 48149 Münster, Germany\\
	  }
	}
	\vspace{2cm}
	\end{center}

	\begin{abstract}
		By recent work of Afandi, it is known that tautological intersection numbers on the moduli space of stable $n$-pointed genus $g$ curves can be arranged into families of Ehrhart polynomials, $\{L_{\vec{d}}\}$, for partial polytopal complexes. In particular, the $f^*$-vector of $L_{\vec{d}}$ is known to be integral and non-negative. In this paper, we show that both the $f^*$-vector and $h^*$-vector have an enumerative interpretation in the special case that $\vec{d} = (1, 1, \dots, 1)$. The $f^*$-vector counts order-consecutive partition sequences of $[n+1]$ and the $h^*$-vector is a binomial coefficient. Furthermore, we conjecture that, for all $\vec{d}$, the $f^*$-vector of $L_{\vec{d}}$ always forms a log-concave sequence, and we verify this conjecture in the case that~$\vec{d} = (1, 1, \dots, 1)$.
	\end{abstract}
	
\end{titlepage}

\section{Introduction}
The moduli space of complex curves has been a fruitful subject of fascinating studies since its construction and compactification by Deligne and Mumford in \cite{deligne1969irreducibility,Mumford1983}. \\
One particular interesting perspective was uncovered recently by Afandi in \cite{afandi2022ehrhart}. Afandi showed that intersection numbers in the moduli space of complex curves are determined by evaluations of Ehrhart polynomials of partial polytopal complexes. By this, it was possible to relate the algebro-geometric investigation of the top degree of the tautological ring of the moduli space of complex curves to combinatorial Ehrhart theory. \\
In this work we investigate the Ehrhart polynomials $\{\Ll_n\}_{n\in\NN}$ corresponding to a restricted class of $\psi$-class intersection numbers
\begin{align*}
	\Ll_n(g) \coloneqq \frac{g!\, 24^g}{n!} \int_{\Mgnbar[g, n+1]} \frac{\psi_1 \cdots \psi_n}{1-\psi_{n+1}}
\end{align*}
We show that two meaningful expansions of $\Ll_n$ are computed by counting order-consecutive partition sequences (see Definitions~\ref{def:consecutive-sequence} and~\ref{def:order-consecutive} in Section~\ref{sec:combinatorial-definitions}) and binomial coefficients, respectively. 
\begin{maintheorem}[Theorem \ref{thm:f*-OCPS} in Section~\ref{sec:fstar}]\label{mainthm:f*-OCPS}
	For $n\in \NN$, the $f^*$-vector $(f^*_i)_{i=0, \dots, n}$ of the normalized Ehrhart polynomial $\Ll_n(g)$ associated to $\psi$-class intersection numbers of powers $\vec{d}=(1, \dots, 1)\in\NN^n$ computes the number of order consecutive partition sequences of $[n+1]$ into $(i+1)$ parts, that is 
	\begin{align*}
		\Ll_n(g) = \sum_{i=0}^{n} f^*_i\; \binom{g-1}{i} \, , && \textnormal{with} && f^*_i = \textnormal{OCPS}^{(n+1)}(i+1) = \sum_{k=0}^{i} (-1)^{i+k} \binom{2k+n}{n} \binom{i}{k}\, .
	\end{align*}
\end{maintheorem}
\begin{maintheorem}[Theorem \ref{thm:h*} in Section~\ref{sec:hstar}]
	In the setting of Theorem~\ref{mainthm:f*-OCPS}, let $n\geq2$. Then, the normalized Ehrhart polynomial $\Ll_n(g)$ enjoys the $h^*$-expansion
	\begin{align*}
		\Ll_n(g)= \sum_{i=0}^n h^*_i \; \binom{g+n-i}{n}\, && \text{with} && h^*_i = \binom{n+1}{2(i-1)}\, .
	\end{align*}
\end{maintheorem}
Furthermore, we show that $(f^*_i)_{i=0, \dots, n}$ and $(h^*_i)_{i=0, \dots, n}$ are logarithmically concave sequences for all $n\in\NN$. 
Logarithmic concavity is shared by many important sequences in various fields in mathematics. While these connections to a wide range of disciplines have shed some light on logarithmic concavity and have given rise to methods of proving it, it is oftentimes a hard task to establish that a sequence exhibits this property. In particular, this is a new combinatorial result for order-consecutive partition sequences.
\begin{maintheorem}[Theorem \ref{thm:logconcavity-OCPS} in Section~\ref{sec:log-concavity-OCPS}]
	For every positive integer $n$ the sequence $(\textnormal{OCPS}^{(n)}(p))_{p=1, \dots, n}$ is logarithmically concave.
\end{maintheorem}
In the proof of this theorem, which constitutes the main technical part of this paper, we also provide proofs for closed expressions for generating functions of the number of order-consecutive partition sequences.\\

The article can be divided into two parts. In the first part, in sections~\ref{sec:EhrhartThy} and~\ref{sec:ModuliSpace}, aspects of classical Ehrhart theory and of the moduli space of complex curves are recalled. In the second part, in Section~\ref{sec:Results}, following the results of \cite{afandi2022ehrhart}, the proofs of the results of this work are presented. Section~\ref{sec:fstar} discusses the $f^*$-expansion, while in Section~\ref{sec:hstar} the $h^*$-expansion is provided. The combinatorial result of logarithmic concavity of the number of order-consecutive partition sequences is shown in Section~\ref{sec:log-concavity}.

\section*{Acknowledgements}
The author F. B. K. wants to thank Adam Afandi for giving the inspiration for this work and introduction to the topics as well as valuable discussions and proofreading. This work is funded by the Deutsche Forschungsgemeinschaft through the Research Training Group 2149 "GRK 2149: Strong and Weak Interactions - from Hadrons to Dark Matter" as well as supported\footnote{funded by the Deutsche Forschungsgemeinschaft (DFG, German Research Foundation) under Germany's Excellence Strategy EXC 2044 -390685587, Mathematics Münster: Dynamics-Geometry-Structure} by the Cluster of Excellence Mathematics Münster.

\section{Ehrhart theory}\label{sec:EhrhartThy}
The exposition of Ehrhart theory presented here is tailored to this paper. For an extensive introduction the textbook by Beck and Robins \cite{BeckSinai} is a good starting point. \\
The basic objects in Ehrhart theory are polytopes. Let $\{v_1, \dots , v_n\}$ be a set of vectors that span $\RR^d$. A \textit{polytope} $P$ is defined by
\begin{align}
	P = \conv[v_1, \dots, v_n] \, .
\end{align}
Given a polytope $P$, a hyperplane $\Hh$ is a supporting hyperplane of $P$, if $P$ lies entirely in one closed half-space defined by $\Hh$. Then, a \textit{face} of $P$ is given by the intersection of $P$ with a supporting hyperplane.
\begin{eg*}
	An important example of a polytope is the standard simplex $\Delta_d$ in $d$ dimensions defined by the $d$ unit vectors and the origin (see Figure~\ref{fig:standard2simplex}). 
\end{eg*}
\begin{figure}[t]
	\centering
	\includegraphics[width=0.6\linewidth]{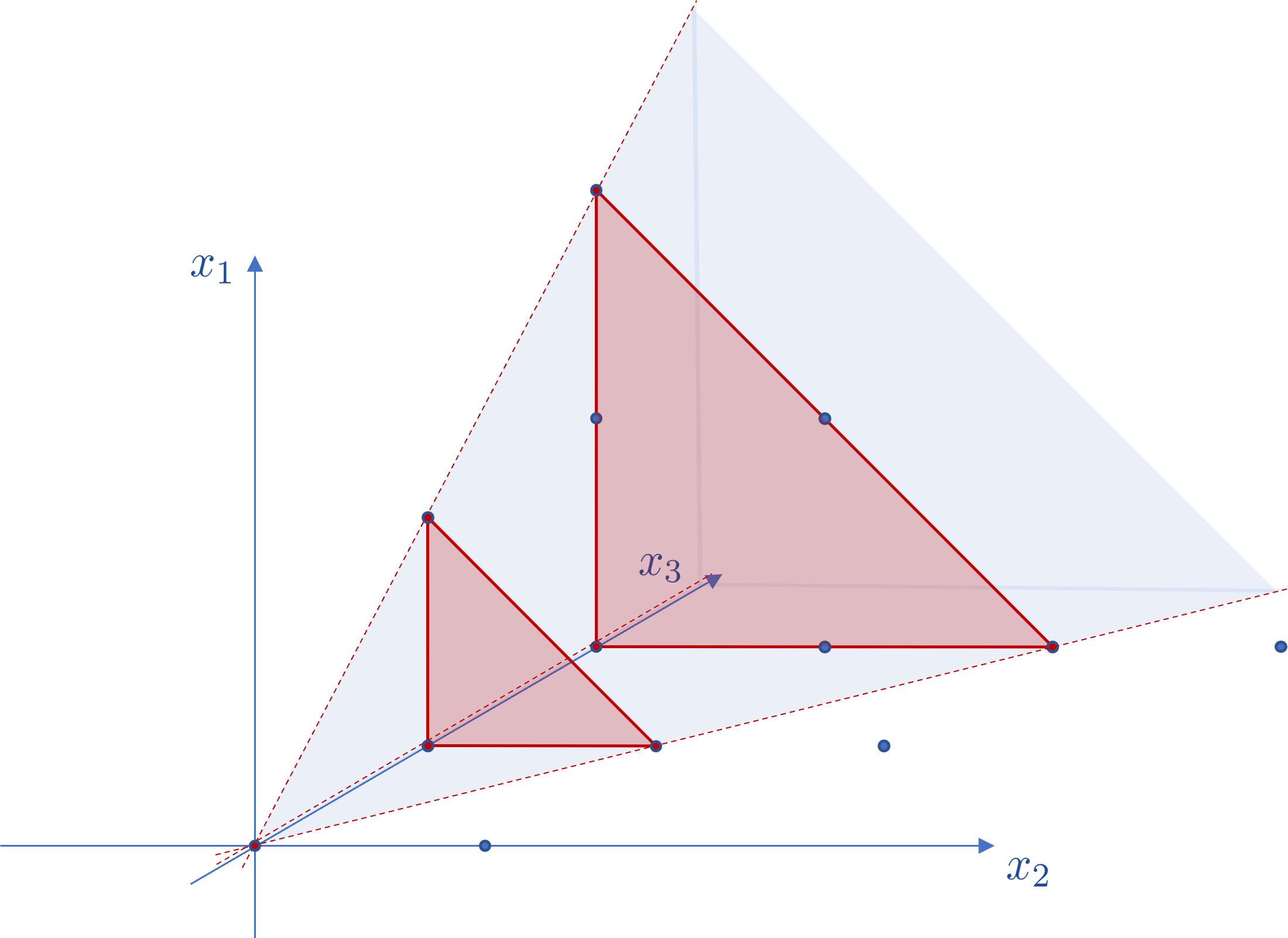}
	\caption{This illustrates the cone (shaded blue) over the standard 2-simplex. Cut out of the cone by the planes $\{x_3=1\}$ as well as $\{x_3=0\}$ and $\{x_3=2\}$ are the standard 2-simplex as well as its zeroth and second dilate, respectively (all shaded red).
	}
	\label{fig:standard2simplex}
\end{figure}

One can construct a cone over a polytope $P$ by embedding $P$ into $\RR^{d+1}$. One maps $\left.v_i\mapsto w_i=(v_i, 1)\right.$, for all $i=1, \dots, n$, and sets $\textnormal{cone}(P)= \textnormal{cone}(w_1, \dots, w_n)$. Then, its \textit{$g$'th dilate} is, for $g\in\NN$, given by 
\begin{align}
	gP \coloneqq \textnormal{cone}(P)\cap\{x_{d+1}=g\}\, . 
\end{align}
If its vertices are in $\ZZ^d$, the polytope and its dilates are called \textit{integral}. Furthermore, an \textit{open} polytope is the relative interior of a polytope. In the first two definitions, the concept of a polytope is generalized.
\begin{definition}
	An \textit{integral polytopal complex} $K$ is a finite collection of integral polytopes containing the empty polytope, such that 
	\begin{itemize}
		\item[$\left.a\right)$] if $Q$ is a face of $P\in K$, then $Q\in K$,
		\item[$\left.b\right)$] if $P, Q\in K$, then $P\cap Q$ is a face of both $P$ and $Q$. 
	\end{itemize}
	The elements of $K$, which are of maximal dimension $d$, are called faces.
\end{definition}
\begin{definition}
	An \textit{integral partial polytopal complex} $K$ of dimension $d$ is the disjoint finite union of open integral polytopes. Again, the elements of $K$ are called faces and are of maximal dimension $d$.
\end{definition}
\begin{remark}
	Note that due to the definitions above, in contrast to polytopal complexes, partial polytopal complexes are not closed under passing to faces, as some relatively open faces might be absent.
\end{remark}
\begin{remark}
	Every polytope can be thought of as a partial polytopal complex, that is as the disjoint union of the interior of its faces. Then, there is a bijection between the faces of the polytope and the partial polytopal complex. To be precise, the faces of the polytope are the relative closures of the faces of the partial polytopal complex. 
\end{remark}

The task of Ehrhart theory is to count the lattice points contained in $gP$. This count is usually encoded in a function $L_P(g)$ named in honor of Ehrhart, who initiated and built this theory \cite{Ehrhart62, Ehrhart77}. As already mentioned above, for a presentation of results obtained in Ehrhart theory and its methods the reader is kindly referred to \cite{afandi2022ehrhart, BeckSinai} and references therein. The first central theorem, which is cited here, was uncovered by Ehrhart himself for polytopes \cite{Ehrhart62} and then generalized.
\begin{theorem}[Ehrhart \cite{Ehrhart62}]\label{thm:Ehrhart}
	The Ehrhart function $L_{K}(g)$ of an integral partial polytopal $d$-complex $K$ is a rational polynomial in $g$ of degree $d$.
\end{theorem}
This result allows applying the theory of rational polynomials to Ehrhart theory. The Ehrhart polynomial $L_K(g)$ of an integral partial polytopal $d$-complex $K$ can be expanded into binomial bases $\{\binom{g-1}{i}\}_{i\in\{0, \dots, d\}}$ or $\{\binom{g+d-i}{d}\}_{i\in\{0, \dots, d\}}$, that is 
\begin{align}\label{equ:fstar-hstar}
	L_K(g) \eqqcolon \sum_{i=0}^d f^*_i\; \binom{g-1}{i}\, , && \text{and} && L_K(g) \eqqcolon \sum_{i=0}^d h^*_i\; \binom{g+d-i}{d}\, .
\end{align}
The associated coefficients are collected in the vectors $f^*, h^*\in \QQ^{d+1}$, respectively. These vectors are the subject of intensive studies in Ehrhart theory and beyond, as information about the partial polytopal complex is encoded in them.\\
In order to interpret the numbers $f^*_i$ and $h^*_i$, a few more notions are needed. A partial polytopal complex $K$ is \textit{simplicial}, if all of its faces are simplices. Furthermore, a \textit{triangulation} of $K$ is a simplicial complex whose support is $K$. Such a triangulation of $K$ is \textit{unimodular}, if the simplices of $K$ can be mapped by an affine automorphism of $\ZZ^d$ to the standard simplex.
\begin{theorem}[{Breuer \cite[Section~2.3]{breuer2012ehrhart}}]\label{thm:fi-counting}
	If $C$ is a unimodular triangulation of a polytopal complex $K$, then the $f^*_i$ count the $i$-dimensional open simplices in $C$. \nopagebreak\\
	If, furthermore, $C$ is a disjoint union of unimodular half-open\footnote{A \textit{half-open} polytope is a set of the form $P\backslash \cup_{i=1}^l p_i$, where $P$ is a polytope	and $p_i$ are faces of $P$. Note that every such polytope is supported by a partial polytopal complex.} simplices of dimension $d$, then the $h^*_i$ count the $i$-dimensional relatively	open unimodular simplices.
\end{theorem}
Note that not every integral polytopal complex can be unimodularly triangulated. In that case the interpretation of the expansion coefficients $f^*_i$ given in Theorem~\ref{thm:fi-counting} does not hold. However, the expansion of $L_P(g)$ in terms of the binomial basis of $\QQ[g]$ remains valid. \\\\
To conclude this section the expansion coefficients collected in the $f^*$- and $h^*$-vector are related. 

\subsection{Relation of $f^*$- and $h^*$-vector}\label{sec:f*-h*-relation}
The $f^*$- and $h^*$-coefficients (see equation~\eqref{equ:fstar-hstar}) of an integer-valued polynomial $p\in\QQ[k]$ of degree $d$ as defined above are generated by 
\begin{align}\label{equ:f*-h*-genfunc}
	\sum_{k\geq0} p(k) z^k = \frac{\sum_{i=0}^d h^*_i z^i}{(1-z)^{d+1}} \eqqcolon H^*(z)\, , && \text{and} && \sum_{k\geq1} p(k) z^k = \sum\nolimits_{i=0}^{d} f^*_i \frac{z^{i+1}}{(1-z)^{i+1}} \eqqcolon F^*(z) \, .
\end{align}
This can be seen as follows. 
\paragraph{$h^*$-expansion} First, expand the denominator and use the reflection relation and symmetry of binomial coefficients to find 
\begin{align}
	\frac{\sum_{i=0}^d h^*_i z^i}{(1-z)^{d+1}} =  \sum_{i=0}^d h^*_i z^i \sum_{k\geq0} \binom{d+k}{k} z^k = \sum_{i=0}^d \sum_{k\geq i} h^*_i \binom{d+k-i}{d} z^k\, ,
\end{align}
where in last step one shifts $k\mapsto k-i$. As the $k$ sum does not change when starting at zero, then, interchange the two sums to find the desired form
\begin{align}
	\sum_{k\geq 0} z^k \sum_{i=0}^d h^*_i \binom{k+d-i}{d} = \sum_{k\geq 0} z^k p(k) \, .
\end{align}

\paragraph{$f^*$-expansion} In a parallel fashion one shows the expression for the $f^*$-vector using 
\begin{align}
	\frac{z^{j+1}}{(1-z)^{j+1}} = \sum_{k\geq1} \binom{k-1}{j} z^k\, .
\end{align}

\begin{remark}
	The expressions~\eqref{equ:f*-h*-genfunc} can already be found in \cite{breuer2012ehrhart}\footnote{The $h^*$-expression is a classical result by Stanley, while the idea of generating functions for the $f^*$- and $h^*$-vector goes already back to Ehrhart.}. There, the definitions are different to the present work. This concerns in particular the binomial coefficient $\binom{n}{m}$ beyond the classical domain $n\geq 0$ and $0\leq m\leq n$. The fact that the combinatorial interpretation breaks down for negative $n$ suggests $\binom{n}{m}=0$ for all $n<0$. This is in conflict with the continuation of the binomial coefficient using the $\Gamma$-function to negative arguments. In order to adapt to this, the sum in~\eqref{equ:f*-h*-genfunc} begins at $k=1$. 
\end{remark}
Using the two generating series defined in equation~\eqref{equ:f*-h*-genfunc} one can relate the $f^*$- and $h^*$-coefficients via $H^*(z) = p(0) + F^*(z)$. Therefore, 
\begin{align}\label{equ:h*-to-f*}
	\sum_{i=0}^d h^*_i z^i &= p(0)(1-z)^{n+1} + \sum_{i=0}^{d} f^*_i z^{i+1}(1-z)^{d-i}\, ,
\end{align}
which can be interpreted as artificially adding the term $f^*_{-1}=p(0)$.

\section{Moduli space of curves}\label{sec:ModuliSpace}
In the following, the moduli space of stable $n$-pointed curves is briefly introduced. For a more comprehensive account of the theory as well as computations, the reader may take a look into references~\cite{Zvonkine23, Pandharipande15, Kock21}. \\
The compactified moduli space $\Mgnbar$ of stable $n$-pointed genus $g$ complex curves is a smooth Deligne-Mumford stack of dimension $d_{g, n} = 3(g-1)+n$. Stability requires the automorphism group of the curves to be finite, which is provided as long as $\chi_{g, n} = 2(1-g)-n < 0$. The interior of $\Mgnbar$ is denoted by $\Mgn$ and contains all smooth curves. The boundary components parametrize nodal degenerations of smooth curves, such that every component, again, is stable. This introduces the well known stratification of $\Mgnbar$. \\\\
Investigating this space oftentimes amounts to studying various classes defined on it and their intersection numbers. A large class of so-called tautological intersection numbers can be reduced to those of $\psi$-classes. To define these, let $(C; p_1, \dots, p_n)$ be a curve in $\Mgnbar$ with marked points $(p_1, \dots, p_n)$. Furthermore, let $\LL_i$ be the line bundle on $\Mgnbar$ whose fibre over $(C; p_1, \dots, p_n)$ is the cotangent space $T^*_{p_i}C$ to the curve $C$ at the $i$'th marked point $p_i$. 
\begin{definition}
	For $1\leq i \leq n$, the \textit{$\psi$-class} $\psi_i = c_1(\LL_i)$ is the first Chern class of the cotangent bundle $\LL_i$. Furthermore, for non-negative integers $\{d_1, \dots, d_n\}$ such that $\sum_{i=1}^n d_i = d_{g, n}$, \textit{intersection numbers of $\psi$-classes} are denoted by
	\begin{align}
		\angles[\tau_{d_1}\dots \tau_{d_n}, {g, n}] = \int_{\Mgnbar} \psi_1^{d_1} \dots \psi_n^{d_n} \in \QQ\, .
	\end{align}
\end{definition}
These numbers are subject to the recursive structure of Virasoro constraints \cite{Kontsevich:1992ti}. The latter are an intricate system of equations that provide a tool to compute all $\psi$-class intersection numbers starting from two base cases of $(g, n)= (0, 3)$ and $(1, 1)$.\\
In this work only intersection numbers with $d_i\equiv 1$ will be considered. Therefore, the Virasoro constraints are not fully recalled here except for the \textit{Dilaton equation}
\begin{align}\label{equ:dilaton-equation}
	\angles[\tau_{d_1}\dots \tau_{d_n}\tau_1, {g, n+1}] = (2(g-1)+n) \angles[\tau_{d_1}\dots \tau_{d_n}, {g, n}]\, .
\end{align} 

\section{On Ehrhart polynomials of unit-power $d_i=1$ intersection numbers}\label{sec:Results}
As already stated in the introduction, it was recently shown in \cite{afandi2022ehrhart} that the information about $\psi$-class intersection numbers is entirely determined by the family of maps
\begin{align*}
	\left\{\begin{array}{ccccc}
		\NN^n && \to && \textnormal{Ehr} \\ \vec{d} && \mapsto && L_{\vec{d}}(g+m)
	\end{array} \right\}_{n\in\NN}\, ,
\end{align*}
where $\textnormal{Ehr}$ denotes the space of Ehrhart polynomials. These maps encode the intersection numbers specified by the array $\vec{d}= (d_1, \dots, d_n)$ in the shifted\footnote{The shift is given by $m= \ceil{(-(n+1)+\sum_{i=1}^n d_i)/3}$, which vanishes for $\vec{d}=(1, \dots, 1)$.} Ehrhart polynomial $L_{\vec{d}}(g+m)$ of a partial polytopal complex. \\
This remarkable result provides an intriguing perspective on tautological intersection numbers and naturally raises the question about what types of partial polytopal complexes correspond to various moduli space data. A partial answer to this for the case of $\vec{d}= (1, \dots, 1)$ is provided in this section in terms of a combinatorial interpretation of the $f^*$- and $h^*$-vector. \\
In this restricted situation the main theorem of \cite{afandi2022ehrhart} reads
\begin{theorem}[{Afandi \cite[Theorem~1]{afandi2022ehrhart}}]\label{thm:adam}
	Let $n\in\NN^\times$. Then there exists a partial polytopal complex $\Pp_n$ of dimension $n$ such that 
	\begin{align}
		24^g (g!) 3^n \angles[(\tau_1)^n  \tau_{3(g-1)+1}, {g}] = L_{\Pp_n}(g)\, ,
	\end{align}
	where $L_{\Pp_n}(g)$ is the integer-valued Ehrhart polynomial of $\Pp_n$. 
\end{theorem}
For details about Theorem~\ref{thm:adam} and its most general version the reader should refer to the original work. There the author provides a detailed and intelligible introduction to the topics as well as further discussion of the result for both algebraic geometers and combinatorialists.
\begin{center}\usebox{\verticalseparator}\end{center}
In the following, the Ehrhart polynomial $L_{\Pp_n}$ is further investigated. Therefore, first, an explicit representation is given.
\begin{proposition}
	In the setting of Theorem~\ref{thm:adam}, let $L_n(g) \coloneqq L_{\Pp_n}(g)$. Then, 
	\begin{align}
		L_n(g) = 3^n \prod_{k=1}^{n} (2(g-1)+k) = \prod_{k=1}^{n} (6g+(3k-6))
	\end{align}
\end{proposition}
\begin{proof}
	This proposition is achieved iteratively using the Dilaton equation~\eqref{equ:dilaton-equation}. Removing the $\psi$-classes step by step one produces in each step one factor of the product in the proposition. The result is completed through the base case $\angles[\tau_{3(g-1)+1}, g]=\frac{1}{g!\,24^g}$.
\end{proof}
In order to find an enumerative interpretation of the $f^*$- and $h^*$-vector of the Ehrhart polynomials $L_n(g)$, an appropriate normalization is provided. Using the representation of $L_n(g)$ found from the algebro-geometric interpretation, define $\Ll_n$ so that
\begin{align}\label{eq:L-norm}
	L_n(g) = 3^n (n!) \Ll_n(g)\, .
\end{align}
\begin{proposition}
	The normalized Ehrhart function $\Ll_n(g)$ defined through equation~\eqref{eq:L-norm} is an integer-valued polynomial. 
\end{proposition}
\begin{proof}
	Note that the coefficients of the polynomial in a monomial basis are in general not divisible by $n!$ as opposed to the $f^*_i$. Therefore, it is suggestive to find an expression in terms of binomial coefficients. This is provided by considering
	\begin{align}
		3^{-n} L_n(g) = \prod_{k=1}^n (2(g-1)+k)&= \frac{(2(g-1)+n)!}{(2(g-1))!} \nonumber\\
		&= n! \cdot \frac{(2(g-1)+n)!}{n!(2(g-1)+n-n)!} = n! \binom{2(g-1)+n}{n}\, .
	\end{align}
\end{proof}
The constructive proof of this proposition explicitly gives 
\begin{align}
	\Ll_n(g) = \binom{2(g-1)+n}{n}\, .
\end{align}

\subsection{$f^*$ expansion}\label{sec:fstar}
In this section it is shown that the $f^*$-vector of the normalized Ehrhart polynomial $\Ll_n(g)$ counts order-consecutive partition sequences (see theorem below). The first subsection~\ref{sec:combinatorial-definitions} introduces the relevant combinatorial definitions and facts, followed by Subsection~\ref{sec:proof-f*-OCPS} with the proof of the theorem.
\begin{theorem}\label{thm:f*-OCPS}
	The $f^*$-vector $(f^*_i)_{i=0, \dots, n}$ of the normalized Ehrhart polynomial $\Ll_n(g)$ associated to $\psi$-class intersection numbers of powers $\vec{d}=(1, \dots, 1)\in \NN^n$ computes the number of order consecutive partition sequences of $[n+1]$ into $(i+1)$ parts, that is 
	\begin{align}
		\Ll_n(g) = \sum_{i=0}^{n} f^*_i\; \binom{g-1}{i} \, , && \textnormal{with} && f^*_i = \textnormal{OCPS}^{(n+1)}(i+1) = \sum_{k=0}^{i} (-1)^{i+k} \binom{2k+n}{n} \binom{i}{k}\, .
	\end{align}
\end{theorem}

\subsubsection{Combinatorial definitions}\label{sec:combinatorial-definitions}
In order to prove this result, the notion of order-consecutive partition sequences is introduced. These combinatorial objects were first investigated in \cite{Chakravarty1985,HWANG1995323}, which the following mainly follows.\\
Start with a $p$-part partition $S^n_p=\{s_1, \dots, s_p\}$ of the ordered set $N_n = \{1, \dots,  n\}$.
\begin{definition}\label{def:consecutive-sequence}
	Let $n, p\in\NN^\times$. A partition $S^n_p$ of the ordered set $N_n$ is a \textit{partition sequence}, if the parts $s_k$ of $S^n_p$ have a fixed order, denoted by $(s_1, \dots, s_p)$.
\end{definition}
\begin{definition}\label{def:order-consecutive}
	In the setting of Definition~\ref{def:consecutive-sequence} a partition sequence is \textit{order-consecutive}, if for all $k=1, \dots , p$ the union $\cup_{i=1}^{k} s_k$ is consecutive. 
\end{definition}

\begin{eg*}
	In order to illustrate the definition above, consider $N_{5}= \{1, \dots, 5\}$ partitioned into three parts by the partition sequence $S=(s_1, s_2, s_3)$ such that
	\begin{align}
		s_1=\{2\}\, , && s_2=\{3, 4\}\, , && s_3 = \{1, 5\}\, .
	\end{align}
	$S$ is order-consecutive as
	\begin{align*}
		\begin{array}{r|ccccc}
			s_1 &&2 &&&\\
			\gray s_1 \cup \black s_2 &&{\gray 2}&3 &4 &\\
			\gray s_1\cup s_2\cup\black s_3 &1 &{\gray 2}&\gray 3&\gray 4&5\, .
		\end{array}
	\end{align*}
	In the same setting, consider as a counterexample $\tilde{S}=(\tilde{s}_1, \tilde{s}_2, \tilde{s}_3)$ defined by 
	\begin{align}
		\tilde{s}_1=\{2\}\, , && \tilde{s}_2=\{3, 5\}\, , && \tilde{s}_3 = \{1, 4\}\, .
	\end{align}
	That this partition sequence is not order-consecutive can be seen through 
	\begin{align*}
		\begin{array}{cr|ccccc}
			&\tilde{s}_1 &&2 &&&\\
			\red \to\;\; &\gray \tilde{s}_1 \cup \black \tilde{s}_2 &&{\gray 2}&3 & \red \_\_ & \black 5 \white\, . \\
			&\gray \tilde{s}_1\cup \tilde{s}_2\cup\black \tilde{s}_3 &1 &{\gray 2}&\gray 3&\black 4&\gray 5\, .
		\end{array}
	\end{align*}
	As indicated, $\tilde{s}_1 \cup \tilde{s}_2 = \{2, 3, 5\}$ is \textit{not} consecutive.
\end{eg*}
Finally, the counting result about order-consecutive partition sequences is stated.
\begin{theorem}[{Hwang-Mallows \cite[Theorem~7]{HWANG1995323}}]\label{thm:OCPS-counting}
	Let $n, p\in\NN^\times$. Then the number of order-consecutive partition sequences of an ordered set with $n$ elements into $p$ parts is 
	\begin{align}
		\textnormal{OCPS}^{(n)}(p) = \sum_{k=0}^{p-1} (-1)^{p-1-k} \binom{p-1}{k} \binom{n-1+2k}{2k}\, .
	\end{align}
\end{theorem}
\textit{Logarithmic concavity} of these numbers is shown in Section~\ref{sec:log-concavity}. \\\\
Computing the first few $\textnormal{OCPS}^{(n)}(p)$ one gets 
\begin{align*}
	\begin{tabular}{c|cccccc}
		\diagbox{$n$}{$p$} &1&2 &3 &4  &5  &6 \\
		\hline
		1&1&  &  &   &   & \\
		2&1&2 &  &   &   & \\
		3&1&5 &4 &   &   & \\
		4&1&9 &16&8  &   & \\
		5&1&14&41&44 &16 & \\
		6&\;\;1\;\;&\;\;20\;\;&\;\;85\;\;&\;\;146\;\;&\;\;112\;\;&\;\;32 
	\end{tabular}
\end{align*}

\subsubsection{Proof of Theorem 4.2}\label{sec:proof-f*-OCPS}
Equipped with the information of the previous section, one can prove Theorem~\ref{thm:f*-OCPS}. This proof relies on the \textit{Gregory-Newton interpolation formula} \cite{newtonformulaFraser1920}. It states that an integer-valued polynomial $p(t)$ in $t$ of degree $d$ can uniquely be expressed in terms of binomials via 
\begin{align}\label{equ:newton-gregory}
	p(t) = \sum_{r=0}^{d} a_r \binom{t}{r}\, ,  && \textnormal{with} && a_r = \sum_{s=0}^{r} (-1)^{r-s} \binom{r}{s} p(s)\, .
\end{align}
\begin{proof}[Proof of Theorem~\ref{thm:f*-OCPS}]
	The assertion is proven in an analytic fashion by means of the Gregory-Newton interpolation formula~\eqref{equ:newton-gregory}. In this context $\Ll_n(g+1) = \binom{2g+n}{n}$ takes the role of $p(g)$. This defines the coefficients $a^{(n)}_r$ via
	\begin{align}
		\binom{2g+n}{n} \eqqcolon \sum_{r=0}^n a_r^{(n)} \; \binom{g}{r}\, .
	\end{align}
	By the interpolation formula the expansion coefficients are determined as
	\begin{align}
		a_r^{(n)} &= \sum_{k=0}^r (-1)^{r-k} \binom{2k+n}{n}\binom{r}{k} = \sum_{k=0}^r (-1)^{r-k} \binom{n+2k}{2k}\binom{r}{k}\, ,
	\end{align}
	using a symmetry of the binomial coefficients. This implies 
	\begin{align}
		\sum_{r=0}^n \left[\sum_{k=0}^r (-1)^{r-k} \binom{n+2k}{2k}\binom{r}{k}\right] \binom{g}{r} = \binom{2g+n}{n} = \Ll_n(g+1)\, .
	\end{align}
	One obtains the assertion of the theorem by shifting $g\mapsto (g-1)$ and Theorem~\ref{thm:OCPS-counting}.
\end{proof}

\subsection{Logarithmic concavity}\label{sec:log-concavity}
In this section logarithmic concavity of the $f^*$ vector of the normalized Ehrhart polynomial $\Ll_n(p)$ is shown. \\
Logarithmic concavity (log-concavity) is a distinguished property of sequences. In the past many important sequences have been shown to have this property with examples throughout various fields, see reviews by Stanley \cite{Stanley89}, Brenti \cite{Brenti94} and Brändén \cite{Braenden14}. 
\begin{definition}
	A sequence of real numbers $(a_n)_{n\in\NN}$ is \textit{log-concave}, if for all $n$ 
	\begin{align}
		(a_n)^2 \geq a_{n+1}a_{n-1}\, .
	\end{align}
\end{definition}
In order to show log-concavity of the number of order-consecutive partition sequences, here, their generating function is investigated. It is a well known result\footnote{This is for instance stated in the introduction of \cite{Braun2016}.} that, if the generating polynomial of some sequence of numbers has only real zeros, the sequence is log-concave.\\
In the following, a proof of a recurrence relation for the count of order-consecutive partition sequences is provided in order to deduce the generating function. Then, the zeros of this generating function will be located by applying an appropriate change of variables.

\subsubsection{Generating function for OCPS}
Arrange the numbers of order-consecutive partition sequences $\text{OCPS}^{(n+1)}(p+1)$ into the generating functions 
\begin{align}
	G(x, y) \coloneqq \sum_{n, p\in \NN} \text{OCPS}^{(n+1)}(p+1)\; x^n y^p\, , && 
	g_n(y) \coloneqq \sum_{p\in\NN} \text{OCPS}^{(n+1)}(p+1)\; y^p\, .
\end{align}
By definition, these two functions are related by 
\begin{align}
	g_n(y) = \frac{1}{n!} \left.\pdv[n]{G(x, y)}{x}\right|_{x=0}\,.
\end{align}
\begin{remark}\label{rem:g-gn}
	$G(x, y)$ is a power series in $x$. Each coefficient of this power series, that is $g_n(y)$, is a polynomial in $y$, since as long as $p>n$ the coefficient of $x^ny^p$ of $G$, that is $\text{OCPS}^{(n+1)}(p+1)$, vanishes.
\end{remark}
The following propositions~\ref{prop:generatingfunction} and~\ref{prop:generatingfunction-g} provide a closed expression for $G$ and $g_n$. 
\begin{proposition}\label{prop:generatingfunction}
	The rational generating function for the numbers of order-consecutive partition sequences $\textnormal{OCPS}^{n+1}(p+1)$ for $n, p\in \NN$ is 
	\begin{align}\label{equ:generatingfunction-OCPS}
		G(x, y) = \frac{1-x}{1 - 2x - 2xy + x^2 + x^2y}\, .
	\end{align}
\end{proposition}
\noindent This generating function expands into
\begin{align}
	G(x, y) \overset{x\to 0}{=} 1+x\left(1 + 2 y\right)+x^2\left(1 + 5y + 4 y^2\right) + \Oo\mleft(x^3\mright)\, .
\end{align}
\begin{remark}
	This generating function $G$ was already conjectured in~\cite{Babusci17} (in Section~9 on page~26). There, the numbers of order-consecutive partition sequences were rediscovered in the context of a combinatorial approach to lacunary series of generalized Laguerre polynomials.
\end{remark}
Proposition~\ref{prop:generatingfunction} can be shown in two steps by first finding a recurrence relation and then deducing the generating function from there. Therefore, the following lemma is given.
\begin{lemma}
	The sequence $\left(\textnormal{OCPS}^{n+1}(p+1)\right)_{n, p\in\NN^\times}$ of numbers of order-consecutive partition sequences is determined by the recursion 
	\begin{align}
		(n)\textnormal{OCPS}^{(n+1)}(p+1) = (2p)\textnormal{OCPS}^{(n)}(p) + (n+2p)\textnormal{OCPS}^{(n)}(p+1)\, ,
	\end{align}
\end{lemma}
\begin{proof}
	This can be seen by explicit calculation
	\begin{align}
		&\textnormal{OCPS}^{(n+1)}(p+1) - \left(\frac{2p}{n}\right) \textnormal{OCPS}^{(n)}(p) - \left(\frac{n+2p}{n}\right)\textnormal{OCPS}^{(n)}(p+1)\nonumber\\
		&= \sum_{k=0}^{p-1} (-1)^{p-k} \binom{p}{k} \left[\binom{n+2k}{2k}-\left(\frac{n+2p}{n}-\frac{2p}{n}\frac{p-k}{p}\right)\binom{n+2k-1}{2k}\right] \nonumber\\
		&\qquad\qquad\qquad\qquad\qquad\quad\;\, +(-1)^{p-p} \binom{p}{p} \left[\binom{n+2p}{2p}-\frac{n+2p}{n}\binom{n+2p-1}{2p}\right]\nonumber\\
		&= \sum_{k=0}^{p-1} (-1)^{p-k} \binom{p}{k} \binom{n+2k}{2k} \left[1-\left(\frac{n+2k}{n}\frac{n}{n+2k}\right)\right]  +  \left[1-\frac{n+2p}{n}\frac{n}{n+2p}\right] =0\, , 
	\end{align}
	where the binomial relation $\binom{a-1}{b} = \frac{a-b}{b} \binom{a}{b}$ was employed several times. 
\end{proof}
Using this lemma Proposition~\ref{prop:generatingfunction} can be proven.
\begin{proof}[Proof of Prop.~\ref{prop:generatingfunction}]
	From the recurrence relation one can deduce a differential equation that the generating function must satisfy. It reads
	\begin{align}
		x \left[1 + (-1 + x)\partial_x + 2 y \left(1 + (1+y)\partial_y\right)\right]h(x, y)=0\, .
	\end{align}
	It can easily be verified that the function $G(x,y)$, see equation~\eqref{equ:generatingfunction-OCPS}, satisfies this differential equation. With the correct initial conditions it, therefore, generates the numbers of order-consecutive partition sequences.
\end{proof}

\begin{proposition}\label{prop:generatingfunction-g}
	For a fixed $n\in \NN$, the generating polynomial for the numbers of order-consecutive partition sequences $\textnormal{OCPS}^{n+1}(p+1)$ for $p\in \NN$ is 
	\begin{align}\label{equ:generatingfunction-g-OCPS}
		g_n(y) = \frac{(1+y)^{\frac{n-1}{2}}}{2} \left[\left(\sqrt{1+y}+\sqrt{y}\right)^{n+1} + \left(\sqrt{1+y}-\sqrt{y}\right)^{n+1}\right]\, .
	\end{align}
\end{proposition}
\begin{remark}
	It is emphasized here, that $g_n(y)$ is in fact a polynomial for each $n\in \NN$ (see Remark~\ref{rem:g-gn}).
\end{remark}
The proof of this proposition is shifted to Appendix~\ref{app:proof-gn} since it turns out to be rather technical.

\subsubsection{Proof of logarithmic concavity of OCPS}\label{sec:log-concavity-OCPS}
The combinatorial Theorem \ref{thm:logconcavity-OCPS}, below, states that the numbers of order-consecutive partition sequences form log-concave sequences. Thereby, it establishes that the $f^*$-vector of the Ehrhart polynomial associated to $\psi$-class intersection numbers of powers one is log-concave.
\begin{theorem}\label{thm:logconcavity-OCPS}
	For every $n\in\NN^\times$ the sequence $(\textnormal{OCPS}^{(n)}(p))_{p=1, \dots, n}$ is log-concave.
\end{theorem}
\begin{proof}
	In this proof, the results of the previous subsection are used. According to Proposition~\ref{prop:generatingfunction-g} the generating function
	\begin{align}
		g_n(y) = \frac{(1+y)^{\frac{n-1}{2}}}{2} \left[\left(\sqrt{1+y}+\sqrt{y}\right)^{n+1} + \left(\sqrt{1+y}-\sqrt{y}\right)^{n+1}\right]
	\end{align}
	generates for a fixed $n\in\NN$ the $\textnormal{OCPS}^{(n+1)}(p+1)$. Thus, to obtain the log-concavity for the sequence of $\textnormal{OCPS}^{(n+1)}(p+1)$ for some fixed $n$, one needs to show for all $n\in \NN$ that $g_n(y)$ has only \textit{real} zeros. In order to do this, note that due to the first factor of its closed expression in the equation above $g_n$ has a zero of order $\ceil{(n-1)/2}$ at $y=-1$. The remaining $\ceil{n/2}$ zeros are determined by 
	\begin{align}\label{equ:zeros-gn}
		\left(\sqrt{1+y}+\sqrt{y}\right)^{n+1} + \left(\sqrt{1+y}-\sqrt{y}\right)^{n+1} = 0\, .
	\end{align}
	This is solved by $y_k\in\Yy_0$, such that 
	\begin{align}\label{equ:zeroset-y0}
		\Yy_0 \coloneqq \left\{ \frac{\tan[2](\frac{\pi}{2}\frac{2k+1}{n+1})}{1 + \tan[2](\frac{\pi}{2}\frac{2k+1}{n+1})} \right\}_{k\in\ZZ}\, .
	\end{align}
	The calculation leading to this presentation of $\Yy_0$ can be found in Appendix~\ref{app:zero-set}. $\Yy_0$ contains $\ceil{n/2}$ different points, due to the symmetry $k\mapsto -(k+1)$. Adding up the numbers of zeros and taking their order into account, one finds $\ceil{n/2}+\ceil{(n-1)/2}=n$. This verifies that all zeros of $g_n$ are found, since $g_n(y)$ is a polynomial of degree $n$ (see Remark~\ref{rem:g-gn}). This concludes the proof since 
	\begin{align}
		\{y\in\CC \colon g_n(y)=0\} = \{-1\} \cup \Yy_0 \subset \RR\, .
	\end{align}
\end{proof}

\subsection{$h^*$ expansion}\label{sec:hstar}
Next to the $f^*$-coefficients of Ehrhart polynomials the $h^*$-coefficients are actually the more classical object to investigate. In this section these coefficients are computed for $\Ll_n(g)$. 

\begin{theorem}\label{thm:h*}
	Let $2\leq n\in\NN$. The normalized Ehrhart polynomial $\Ll_n(g)$, which computes intersection numbers of $\psi$-classes with powers $\vec{d}=(1, \dots, 1)\in\NN^n$, enjoys the $h^*$-expansion
	\begin{align}
		\Ll_n(g)= \sum_{i=0}^n h^*_i \; \binom{g+n-i}{n}\, && \text{with} && h^*_i = \binom{n+1}{2(i-1)}\, .
	\end{align}
\end{theorem}
\begin{proof}
	As stated in Section~\ref{sec:f*-h*-relation} the $h^*$- and $f^*$-vector satisfy the algebraic relation~\eqref{equ:h*-to-f*}. First, note that constant term of $\Ll_n(g)$ is given by $\prod_{k=1}^{n}(k-2)$ which vanishes as long as $n\geq 2$. Then, by expanding the right-hand-side of~\eqref{equ:h*-to-f*} one finds
	\begin{align}
		\sum_{j=0}^{n} f^*_j \frac{z^{j+1}}{(1-z)^{j-n}} = \sum_{j=0}^{n} \sum_{k=0}^{n-j} (-1)^k f^*_j \binom{n-j}{k} z^{j+k+1}\, .
	\end{align}
	Due to the identification of the $f^*$-vector with the number of order-consecutive partition sequences in Theorem~\ref{thm:f*-OCPS}, the $h^*$-vector of $\Ll_n(g)$ is given by the coefficients of 
	\begin{align}
		H^*(z) = \sum_{i=0}^n h^*_{i} z^i = \sum_{j=0}^{n} \sum_{i=0}^{n-j} \sum_{l=0}^j (-1)^{j-l+k} \binom{n+2l}{2l} \binom{j}{l} \binom{n-j}{k} z^{j+k+1}\, .
	\end{align}
	Reordering of the sums yields 
	\begin{align}
		h^*_{i+1} &= \sum_{l=0}^{i} \sum_{j=l}^{i} (-1)^{i-l} \binom{n+2l}{2l} \binom{j}{l} \binom{n-j}{i-j} \nonumber\\
		&= \sum_{l=0}^i (-1)^{i-l} \binom{n+2l}{2l}\sum_{j=l}^{i} \binom{l+j}{l} \binom{n-l-j}{i-l-j} \, .
	\end{align}
	The inner sum can be treated by the Chu-Vandermonde relation in the transformed form $\sum_{K=0}^{N}\binom{X+K}{K}\binom{Y+N-K}{N-K} = \binom{X+Y+N+1}{N}$, with $N=k-l$, $X=l$ and $Y=d-k$. It can be deduced from the more-common version by the relation $(-1)^v \binom{u}{v} = \binom{-u+v-1}{v}$. Thus, 
	\begin{align}
		h^*_{i+1} = \sum_{l=0}^i (-1)^{i-l} \binom{n+2l}{2l} \binom{n+1}{i-l} \, .
	\end{align}
	Rewriting in terms of the $\Gamma$-function, one has 
	\begin{align}
		h^*_{i+1} = \frac{\Gamma(-n-1+2i)}{\Gamma(-n-1)\Gamma(2i+1)}\, .
	\end{align}
	This computes the desired result which can be seen using the reflection identity of the $\Gamma$-function $\frac{\Gamma(-u+v+1)}{\Gamma(-u)}= (-1)^{v+1} \frac{\Gamma(u+1)}{\Gamma(u-v)}$, with $u=n+1$ and $v=2i-1$.
\end{proof}
At this point a remark on the restriction of the theorem above to $n\geq2$ is appropriate. As mentioned in the proof, this is due to the fact that as long as $n\geq2$, the Ehrhart polynomial $\Ll_n(g)$ has no constant term. For the two exceptional cases 
\begin{align}
	\Ll_0(g) = 1\, , && \text{and} && &\Ll_1(g) = 2g-1\, ,
\intertext{one calculates by hand}
	n=0: \quad h^*_0 = 1\, , && \text{and} && &n=1: \quad h^*_0 = -1\, , \; h^*_1 = 3\, .
\end{align}
The distinction of $n=0, 1$ here can be interpreted using the algebro-geometric origin of the Ehrhart polynomials. Recalling that $L_n$ computes $\psi$-class intersection numbers on $\Mgnbar[g, n+1]$, the Ehrhart polynomials indexed by $n=0, 1$ touch the realm of unstable topologies of positive Euler characteristic $\chi_{g, n}$. Here intersection numbers are classically not well-defined. \\
Furthermore, it is notable that starting from $n=2$ the $h^*$-vector is non-negative and $h^*_0$ always vanishes. Parallel to the $f^*$-vector, the $h^*$-vector has nice properties, which can directly be read off from the representation in Theorem~\ref{thm:h*} such as the fact that the sequences $\{(h^*_i)_{i=0, \dots, n}\}_{n\in\NN}$ have no internal zeros and are log-concave.\\

To summarize, for the first few $\Ll_n$ the $h^*_i$ are given by 
\begin{align*}
	\begin{tabular}{c|cccccc}
		\diagbox
		{$n$}{$i$} &0&1 &2 &3  &4  &5 \\
		\hline
		1&1&  &  &   &   & \\
		2&$-1$&3 &  &   &   & \\
		3&0&1 &3 &   &   & \\
		4&0&1 &6&1  &   & \\
		5&0&1&10&5 &0 & \\
		6&\;\;0\;\;&\;\;1\;\;&\;\;15\;\;&\;\;15\;\;&\;\;1\;\;&\;\;0 
	\end{tabular}
\end{align*}

\section{General Conjecture}\label{sec:general-conjecture}
In this paper we have analyzed the subsector of $\psi$-class intersection numbers of the form 
\begin{align}
	\int_{\Mgnbar[g, n+1]} \frac{\psi_1^{d_1} \cdots \psi_n^{d_n}}{1-\psi_{n+1}} && \text{with} && \vec{d} = (1, \dots, 1)\, .
\end{align}
We conjecture, however, that logarithmic concavity of the $f^*$- and $h^*$-vector generalizes to all $\psi$-class intersection numbers.
\begin{mainconjecture}
	The Ehrhart polynomial $L_{C}$ of a partial polytopal complex $C$ associated to $\psi$-class intersection numbers specified by $d = (d_1, \dots, d_n)\in\NN^n$ for $n\in\NN$, given by 
	\begin{align}
		L_{C}(g) = (g+m)!24^{g+m} \prod_{i=1}^{n}(2d_i+1)!! \int_{\Mgnbar[g, n+1]} \frac{\psi_1^{d_1} \cdots \psi_n^{d_n}}{1-\psi_{n+1}}\, ,
	\end{align}
	enjoys logarithmically concave $f^*$- and $h^*$-expansions.
\end{mainconjecture}
In Appendix~\ref{app:general-conjecture} we provide data supporting this statement for $\vec{d}$ beyond the case $\vec{d}=(1, \dots, 1)$.

\section{Conclusion}
The moduli space of curves has been a source of inspiration for numerous projects in the past, resulting in significant advancements in its theory and beyond. In this work the Ehrhart polynomial associated to intersection numbers on the moduli space of curves of $\psi$-classes with unit power are investigated, building on \cite{afandi2022ehrhart} by Afandi. \\
There it was shown that tautological intersection numbers on the moduli space of curves are computed by Ehrhart polynomials of partial polytopal complexes. The present work initiates studies to understand which geometries in the vast class of partial polytopal complexes actually correspond to the algebro-geometric problem of tautological intersection numbers. \\\\
Therefore, the $f^*$- and $h^*$-expansion of the Ehrhart polynomials for the aforementioned subclass of intersection numbers is calculated, and combinatorial interpretations are found. The $f^*$-vector, in particular, counts order-consecutive partition sequences, while the $h^*$-vector computes binomial coefficients. Being related to an exceptionally nice problem on the moduli space of curves, it should not be surprising that these numbers satisfy meaningful properties. It is shown that the $f^*$- and $h^*$-vector investigated here are logarithmically concave, which is a deep property inherent to many important sequences across mathematics. \\\\
This further characterizes the locus of polytopal objects corresponding tautological intersection numbers on the moduli space of curves. Therefore, it constitutes a step towards explicitly constructing them from the algebro-geometric data. Furthermore, this work has shown, that already on the way interesting objects can be rediscovered and analyzed.

\begin{appendices}

\section{Proof of Proposition 4.4}\label{app:proof-gn}
Here the closed form expression for the generating polynomial of $\text{OCPS}^{(n+1)}(p+1)$ for fixed $n$ is proved. This can be done using elementary methods via a tedious induction in $n$. The proof presented here, however, uses relations of special functions named after Chebyshev. The author wants to thank R. Wulkenhaar for pointing out this elegant approach.
\begin{proof}[Proof of Proposition~\ref{prop:generatingfunction-g}]
	The generating polynomial $g_n(y)$ is the coefficient of $x^n$ in the series $G(x, y)$. In order to extract this coefficient rewrite 
	\begin{align}
		G(x, y) = \frac{1-x}{1-2x (1+y)+x^2(1+y)} =\frac{1-x}{1-2\sqrt{1+y} \left(x \sqrt{1+y}\right)+\left(x\sqrt{1+y}\right)^2}\, .
	\end{align}
	This can be expanded into 
	\begin{align}
		G(x, y) = (1-x)\sum_{n=0}^\infty U_n\mleft(\sqrt{1+y}\mright) (1+y)^{n/2}x^n\, ,
	\end{align}
	where $U_n(t) = C^{(1)}_n(t)$ are Chebyshev polynomials of the second kind, a special case of Gegenbauer polynomials $C^{\lambda}_n$ \cite{Stein71}. These polynomials have explicit representations, which will be employed here. For $\abs{t}\leq 1$, let $t = \cos \phi$, then
	\begin{align}
		U_n(\cos\phi) = \frac{\sin((n+1)\phi)}{\sin\phi}\, .
	\end{align} 
	Using this as well as trigonometric relations, one calculates 
	\begin{align}
		[x^n]   G(x,t^2-1) &= t^n U_n(t)  - t^{n-1} U_{n-1}(t)\nonumber\\
		&= \frac{\sin((n+1)\phi)}{\sin \phi} \cos^n \phi- \frac{\sin(n\phi)}{\sin \phi} \cos^{n-1} \phi \nonumber\\
		&= \left[ \frac{\sin(n\phi)}{\sin\phi} \cos^2\phi + \cos\phi\cos(n\phi) - \frac{\sin(n\phi)}{\sin\phi} \right] \cos^{n-1} \phi\nonumber\\
		&= \left[ \cos\phi \cos(n\phi) - \sin\phi\sin(n\phi) \right] \cos^{n-1} \phi\nonumber\\
		&= \cos((n+1)\phi) \cos^{n-1} \phi = t^{n-1} T_{n+1}(t)\, , 
	\end{align}
	where $T_n$ are the Chebyshev polynomials of first kind \cite{tchebychev1853}. These, in turn, can be written as 
	\begin{align}
		T_n(t)= \frac{1}{2} \left[\left(t+\sqrt{t^2-1}\right)^n +\left(t-\sqrt{t^2-1}\right)^n\right]\, .
	\end{align}
	Passing to $t^2-1=y$ one finds 
	\begin{align}\label{equ:app-gn}
		g_n(y) = \frac{(1+y)^{\frac{n-1}{2}}}{2} \left[\left(\sqrt{1+y}+\sqrt{y}\right)^{n+1} + \left(\sqrt{1+y}-\sqrt{y}\right)^{n+1}\right]\, .
	\end{align}
	Note that for $\abs{t}\geq1$, there are equivalent relations of the Chebyshev polynomials to hyperbolic trigonometric functions giving the same result for $g_n$ as equation~\eqref{equ:app-gn}. This concludes the proof.
\end{proof}

\section{Solution of equation~\eqref{equ:zeros-gn} in Section~\ref{sec:log-concavity-OCPS}
}\label{app:zero-set}
In the proof of Theorem~\ref{thm:logconcavity-OCPS} it is claimed that equation~\eqref{equ:zeros-gn}, that is
\begin{align}\label{equ:app-zeros-gn}
	\left(\sqrt{1+y}+\sqrt{y}\right)^{n+1} = - \left(\sqrt{1+y}-\sqrt{y}\right)^{n+1}\, ,
\end{align}
is solved by the elements of 
\begin{align}
	\Yy_0 = \left\{ \frac{\tan[2](\frac{\pi}{2}\frac{2k+1}{n+1})}{1 + \tan[2](\frac{\pi}{2}\frac{2k+1}{n+1})} \right\}_{k\in\ZZ}\, .
\end{align}
In order to see this, note that equation~\eqref{equ:app-zeros-gn} is equivalent to 
\begin{align}
	\sqrt{1+y_k}+\sqrt{y_k} = e^{\frac{\pi \mathi}{n+1} + \frac{2\pi\mathi k}{n+1}}\left(\sqrt{1+y_k}-\sqrt{y_k}\right)\, ,
\end{align}
for all $k\in \ZZ$. Rearranging this yields 
\begin{align}
	\sqrt{\frac{1+y_k}{y_k}} = \frac{\exp(\mathi \pi \frac{2k+1}{n+1})-1}{\exp(\mathi \pi \frac{2k+1}{n+1})+1} = \frac{\exp(\frac{\mathi\pi}{2} \frac{2k+1}{n+1})-\exp(- \frac{\mathi\pi}{2} \frac{2k+1}{n+1})}{\exp(\frac{\mathi\pi}{2} \frac{2k+1}{n+1})+\exp(-\frac{\mathi\pi}{2} \frac{2k+1}{n+1})} = \mathi \tan(\frac{\pi}{2} \frac{2k+1}{n+1})\, . 
\end{align}
In the last step Euler's formula $\exp[\mathi z] = \cos(z) + \mathi \sin(z)$ was employed to reduce the numerator and denominator to a sine and cosine, respectively, giving a tangent. Finally, solving for $y_k$ gives the desired result 
\begin{align}
	y_k = \frac{\left(\mathi \tan(\frac{\pi}{2} \frac{2k+1}{n+1})\right)^2}{\left(\mathi \tan(\frac{\pi}{2} \frac{2k+1}{n+1})\right)^2-1} = \frac{\tan[2](\frac{\pi}{2} \frac{2k+1}{n+1})}{\tan[2](\frac{\pi}{2} \frac{2k+1}{n+1}) + 1}\, .
\end{align}
By closely examining this expression one realizes that symmetries of the tangent translate to the identification $y_k = y_{-k-1}$. Thus, the number of different values in $\Yy_0$ reduces by half, to be precise to $\ceil{n/2}$.

\section{Data supporting general conjecture}\label{app:general-conjecture}
In Section~\ref{sec:general-conjecture} we conjecture that log-concavity of the $f^*$- and $h^*$-vector of the Ehrhart polynomial associated to intersection numbers of the form 
\begin{align}
	\int_{\Mgnbar[g, n+1]} \frac{\psi_1^{d_1} \cdots \psi_n^{d_n}}{1-\psi_{n+1}}
\end{align}
holds true for general $\vec{d}\in\NN^n$ for $n\in\NN$ beyond the case of $\vec{d} = (1, \dots, 1)$, which was treated in this work. Here, we provide numerical data supporting this conjecture presenting a small subset of data that is collected. \\
In Table~\ref{tab:fstarhstar} we list the $f^*$- and $h^*$-vectors of the Ehrhart polynomials corresponding to the intersection numbers specified by the following vectors 
\begin{align}\label{equ:d-vecs}
	\vec{d} \in \{(2, 1, 1), \; (2, 2, 1), \; (2, 2, 2), \; (5, 1, 1), \; (5, 5, 1), \; (5, 5, 5)\}\, .
\end{align}
In the logarithmic plots in Figure~\ref{fig:fstar-hstar}, one can observe the concave shape of the $f^*$- and $h^*$-vectors associated to the $\vec{d}$ listed above. 
\begin{figure}[t]
	\centering
	{\includegraphics[width=0.35\linewidth]{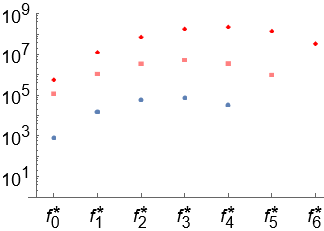}
	\hspace{1.4cm}
	\includegraphics[width=0.35\linewidth]{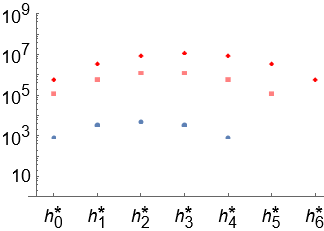}}
	{\includegraphics[width=0.35\linewidth]{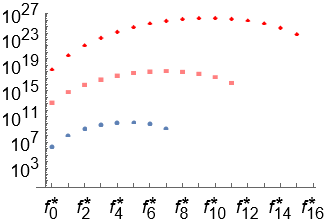}
	\hspace{1.4cm}
	\includegraphics[width=0.35\linewidth]{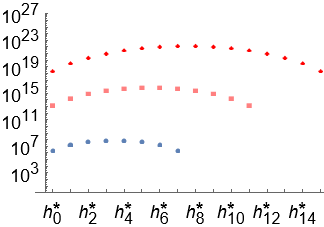}}
	\caption{This depicts the on the left the $f^*$- and on the right the absolute value of the $h^*$-coefficients associated to vectors $\vec{d}$ listed in equation~\eqref{equ:d-vecs} on a logarithmic scale. The upper panels depict the coefficients corresponding to $(2, 1, 1)$ (blue bullets), $(2, 2, 1)$ (pink squares), and $(2, 2, 2)$ (red diamonds) -- the lower panels those to $(5, 1, 1)$ (blue bullets), $(5, 5, 1)$ (pink squares), and $(5, 5, 5)$ (red diamonds)}
	\label{fig:fstar-hstar}
\end{figure}
In order to verify log-concavity numerically, that is if 
\begin{align}
	a_i^2 \geq a_{i-1}a_{i+1}\, , 
\end{align}
for $a\in\{f^*, h^*\}$ for all $i\in \NN$, define\footnote{\label{fn:gamma-ratio}The ratio $\gamma_i^{(a)}$ is well-defined as long as $a_i\neq 0$. As the sequences considered here have no internal zeros, one can therefore calculate $\gamma^{(a)}_i$ as long as $i=0, \dots, i_*$, where $i_*$ is the highest $i$ such that $a_i\neq 0$. In order to define $\gamma_i$ at the boundary of the classical domain, we set $f^*_i = h^*_i = 0$ for negative $i$ and if $i>i_*$.}
\begin{align}\label{equ:gamma-factor}
	\gamma^{(a)}_i = a_{i-1}a_{i+1}/a_i^2
	\, .
\end{align}
If the sequence $(a_i)_{i\in\NN}$ is log-concave, then $\gamma^{(a)}_i \leq 1$ for all $i\in\NN$. As presented in Table~\ref{tab:frathrat}, this is true for all vectors $\vec{d}$ considered here. 
\begin{table}[H]
	\centering
	\begin{tabular}{c||p{6.1cm}|p{6.1cm}}
		$\vec{d}$ & \centering$f^*$ & $\qquad\qquad\qquad\quad h^*$ \\
		\hline\hline
		(2, 1, 1) & $(810, 14850, 56808, 73872, 31104)$ & $(810, -3240, 4860, -3240, 810)$ \\
		\hline
		(2, 2, 1) & $(113400, 1043280, 3363768, $ \newline$4984416, 3483648, 933120)$ & $(113400, -567000, 1134000, $ \newline$-1134000, 567000, -113400)$ \\
		\hline
		(2, 2, 2) & $(567000, 12312000, 70201080, $ \newline$175718160, 220838400, 137168640, $ \newline$33592320)$ & $(567000, -3402000, 8505000, $ \newline$-11340000, 8505000, -3402000, $ \newline$567000)$ \\
		\hline
		(5, 1, 1) & $(1871100, 111068496, $ \newline$1181882988, 4917186648, $ \newline$10150059456, 11142852480, $ \newline$6248171520, 1410877440)$ & $(1871100, -13097700, $ \newline$39293100, -65488500, $ \newline$65488500, -39293100, $ \newline$13097700, -1871100)$ \\
		\hline
		(5, 5, 1) & $(12795710447520, $ \newline$545672898650880, $ \newline$7514237614403520, $ \newline$51028612339985280, $ \newline$203286770565167520, $ \newline$517792118624161920, $ \newline$880212456654612480, $ \newline$1013640585479262720, $ \newline$783018965116846080, $ \newline$389338243761438720, $ \newline$112781704966963200, $ \newline$14481697524940800)$ & $(12795710447520, $ \newline$-140752814922720, $ \newline$703764074613600, $ \newline$-2111292223840800, $ \newline$4222584447681600, $ \newline$-5911618226754240, $ \newline$5911618226754240, $ \newline$-4222584447681600, $ \newline$2111292223840800, $ \newline$-703764074613600, $ \newline$140752814922720, $ \newline$-12795710447520)$ \\
		\hline
		(5, 5, 5) & $\left(1935093730525956000,\right. $ \newline$313789090270595940000, $ \newline$ 11148733407647723025600,$ \newline$ 170915024328944674363200, $ \newline$1457161238790939517466400,$ \newline$ 7848475852338743842836000, $ \newline$28721323128140529351696000,$ \newline$ 74592064602862919259840000, $ \newline$141024211328158390374912000,$ \newline$ 196484612524240769351884800, $ \newline$201914238092811385946112000,$ \newline$ 151331816708387046014976000, $ \newline$80505819140360989468262400,$ \newline$ 28821324615457982172364800, $ \newline$6230468372491573552742400,$ \newline$\left.614848852548510547968000\right)$ & $(1935093730525956000, $ \newline$-29026405957889340000, $ \newline$203184841705225380000, $ \newline$-880467647389309980000, $ \newline$2641402942167929940000, $ \newline$-5811086472769445868000, $ \newline$9685144121282409780000, $ \newline$-12452328155934526860000, $ \newline$12452328155934526860000, $ \newline$-9685144121282409780000, $ \newline$5811086472769445868000, $ \newline$-2641402942167929940000, $ \newline$880467647389309980000, $ \newline$-203184841705225380000, $ \newline$29026405957889340000, $ \newline$-1935093730525956000)$ 
	\end{tabular}
	\captionof{table}{Here the $f^*$- and $h^*$-vectors associated to vectors $\vec{d}$ listed in equation~\eqref{equ:d-vecs} are presented.}
	\label{tab:fstarhstar}
\end{table}
\begin{table}[H]
	\centering
	\begin{tabular}{c||p{6.5cm}|p{6cm}}
		$\vec{d}$ & \centering $\gamma^{(f^*)}$ & { $\qquad\qquad\qquad\quad\gamma^{(h^*)}$} \\
		\hline\hline
		(2, 1, 1) & $(0, 0.2087, 0.0003, 7.9\cdot10^{-8}, 0)$ & $(0, 0.375, 0.4444, 0.375, 0)$ \\
		\hline
		(2, 2, 1) & $(0, 0.3505, 0.0001, 9.1\cdot10^{-9}, $ \newline$4.5\cdot10^{-13}, 0)$ & $(0., 0.4, 0.5, 0.5, 0.4, 0)$ \\
		\hline
		(2, 2, 2) & $(0., 0.2626, 3.0\cdot10^{-4}, 7.0\cdot10^{-10}, $ \newline$6.4\cdot10^{-15}, 3.7\cdot10^{-20}, 0)$ & $(0., 0.4167, 0.5333, 0.5625, 0.5333, $ \newline$0.4167, 0.)$ \\
		\hline
		(5, 1, 1) & $(0, 0.1793, 6.1\cdot10^{-7}, 2.5\cdot10^{-13}, $ \newline$2.6\cdot10^{-20}, 1.3\cdot10^{-27}, 4.5\cdot10^{-35}, 0)$ & $(0, 0.4286, 0.5556, 0.6, 0.6, 0.5556, $ \newline$0.4286, 0)$ \\
		\hline
		(5, 5, 1) & $(0, 0.3229, 6.4\cdot10^{-9}, 1.1\cdot10^{-17}, $ \newline$3.1\cdot10^{-27}, 2.2\cdot10^{-37}, 6.4\cdot10^{-48}, $ \newline$1.1\cdot10^{-58}, 1.5\cdot10^{-69}, 2.5\cdot10^{-80},$ \newline$ 6.2\cdot10^{-91}, 0)$ & $(0, 0.4545, 0.6, 0.6667, 0.7, 0.7143, $ \newline$0.7143, 0.7, 0.6667, 0.6, 0.4545, 0)$ \\
		\hline
		(5, 5, 5) & $(0., 0.2191, 3.3\cdot10^{-12}, 1.8\cdot10^{-24}, $ \newline$7.2\cdot10^{-38}, 3.7\cdot10^{-52}, 3.6\cdot10^{-67}, $ \newline$1.0\cdot10^{-82}, 1.1\cdot10^{-98}, 6.7\cdot10^{-115}, $ \newline$2.5\cdot10^{-131}, 9.6\cdot10^{-148}, 4.7\cdot10^{-164}, $ \newline$3.9\cdot10^{-180}, 6.6\cdot10^{-196}, 0)$ & $(0, 0.4667, 0.6190, 0.6923, 0.7333, $ \newline$0.7576, 0.7714, 0.7778, 0.7778, $ \newline$0.7714, 0.7576, 0.7333, 0.6923, $ \newline$0.6190, 0.4667, 0)$
	\end{tabular}
	\captionof{table}{Here the ratios $\gamma^{(f^*)}$ and $\gamma^{(h^*)}$ defined in equation~\eqref{equ:gamma-factor} associated to vectors $\vec{d}$ listed in equation~\eqref{equ:d-vecs} are presented (see also footnote~\ref{fn:gamma-ratio}).}
	\label{tab:frathrat}
\end{table}

\end{appendices}

\addcontentsline{toc}{section}{References}

\bibliographystyle{JHEP}
\bibliography{paper_-_logconcavity_bib}

\end{document}